\documentclass[11pt]{article}

\usepackage[utf8]{inputenc}
\usepackage{amsmath, amssymb, amsfonts, amsthm}
\usepackage{mathtools}
\usepackage{graphicx,xcolor}
\usepackage{enumitem}
\usepackage{tikz-cd}
\usepackage[margin=1in]{geometry}
\usepackage[hang,flushmargin]{footmisc}
\usepackage{setspace}

\DeclarePairedDelimiter{\ceil}{\lceil}{\rceil}

\newtheorem{theorem}{Theorem}[section]
\newtheorem{definition}[theorem]{Definition}
\newtheorem{proposition}[theorem]{Proposition}

\newtheorem{lemma}[theorem]{Lemma}

\usetikzlibrary{arrows}
\title{Structure-biased Maker-Breaker Games}
\author{
  Wesley Pegden\thanks{Department of Mathematical Sciences, Carnegie Mellon University. Email: \texttt{wes@math.cmu.edu}}
  \and
  Francesca Yu\thanks{Department of Mathematical Sciences, Carnegie Mellon University. Email: \texttt{francesca.yu@berkeley.edu}}
}

\date{}
\begin{document}

\maketitle
\begin{abstract}
In classical Maker-Breaker games on graphs, Maker and Breaker take turns claiming edges; Maker's goal is to claim all of some structure (e.g., a spanning tree, Hamilton cycle, etc.), while Breaker aims to stop her.  The standard question considered is how powerful a Breaker Maker can defeat; i.e., for the $(1:b)$-biased game where Breaker takes $b$ edges per turn, how large can $b$ be for Maker to still have a winning strategy, for various possible goal sets?

We introduce a variant of this question in which Breaker is required to choose their multiple edges as the edges of (a subgraph of) a given structure (e.g., a matching, clique, etc.) on each turn.   We establish the order of magnitude of the threshold biases for triangle games, connectivity games, and Hamiltonicity games under clique, matching, and star biases respectively. We conclude that in many cases structure imposes major obstruction to Breaker, opening up a set of games whose strategies deviate from the classical biased Maker-Breaker game strategies, and shedding light on the types of Breaker strategies that may or may not work to prove tighter bounds in the classical setting.
\end{abstract}

\section{Introduction}
The classical $(1:b)$-biased Maker-Breaker game has the set of edges of $K_n$ as its board. The two players, Maker and Breaker, take turns to take edges from the board, with Maker allowed to take $1$ edge for each turn while Breaker is allowed to take $b$ edges for each turn. Maker wins the game if she achieves a certain goal set before the game ends (i.e. when there is no more available edge), otherwise Breaker wins.

Among the classical $(1:b)$-biased Maker-Breaker games, the most well-studied ones include:
\begin{enumerate}
\item $\mathcal{F}_{K_3}(b,n)$: triangle game, where the winning set for Maker consists of all triangles in $K_n$;
\item $\mathcal{C}(b,n)$: connectivity game, where the winning set for Maker consists of all spanning trees of $K_n$;
\item $\mathcal{H}(b,n)$: Hamiltonicity game, where the winning set for Maker consists of Hamiltonian paths in $K_n$;
\end{enumerate}

Given $n$ a large natural number, the threshold bias for a biased Maker-Breaker game is the $b_0(n)$ such that if $b>b_0$, Breaker wins the $(1:b)$-biased game, while if $b\leq b_0$, Maker wins the $(1:b)$-biased game. The known bounds for the threshold biases of the three games above are as follows:

\begin{theorem}[Glazik and Srivastav \cite{glazik2022}] The threshold bias $b$ for $\mathcal{F}_{K_3}(b,n)$ satisfies $$1.414\sqrt{n}\leq b\leq 1.633\sqrt{n}$$
\end{theorem}

\begin{theorem}[Gebauer and Szab\'o \cite{gebauer2009}]
\label{classic-con}
The threshold bias for $\mathcal{C}(b,n)$ is $(1+o(1))\frac{n}{\ln n}$.
\end{theorem}

\begin{theorem}[Krivelevich \cite{kriv2011}]
\label{classic-ham}
The threshold bias for $\mathcal{H}(b,n)$ is $(1+o(1))\frac{n}{\ln n}$.
\end{theorem}
In this paper we introduce the following variant of the biased Maker-Breaker game on cliques: the setup of the game is the same as the classical biased Maker-Breaker game, except that on each turn Breaker must choose the edge set of a given structure or a subgraph of the structure, rather than being able to choose any $b$ unclaimed edges. We call this new variant \emph{structure-biased Maker-Breaker game}.

In Section 3, we define $(1:K_m)$-biased Maker-Breaker games, where the structure restriction on Breaker is the complete graph on $m$ vertices. We determine the order of magnitude for threshold biases of the corresponding three types of games:

\begin{theorem}
\label{big1}
Let $K_m$ denote the complete graph on $m$ vertices, then
\begin{enumerate}
\item The threshold bias for $\mathcal{F}_{K_3}(K_m,n)$ is $m=\Theta(\sqrt{n})$;
\item The threshold bias for $\mathcal{C}(K_m,n)$ is $m=\Theta(\sqrt{\frac{n}{\ln n}})$;
\item The threshold bias for $\mathcal{H}(K_m,n)$ is $m=\Theta(\sqrt{\frac{n}{\ln n}})$;
\end{enumerate}
\end{theorem}

In Section 4, we define $(1:E_b)$-biased Maker-Breaker games, where the structure restriction is matchings of size $b$. We have the following:

\begin{theorem}
\label{big2}
Let $E_b$ denote the matching of size $b$, then
\begin{enumerate}
\item The threshold bias for $\mathcal{F}_{K_3}(E_b,n)$ is $b=\Theta(n)$;
\item The threshold bias for $\mathcal{C}(E_b,n)$ is $b=\Theta(n)$;
\item The threshold bias for $\mathcal{H}(E_b,n)$ is $b=\Theta(n)$;
\end{enumerate}
\end{theorem}

In Section 5, we consider $(1:S_b)$-biased Maker-Breaker games, where the structure restriction on Breaker is stars with $b$ leaves. We establish:

\begin{theorem}
\label{big3}
Let $S_b$ denote the star with $b$ leaves, then
\begin{enumerate}
\item The threshold bias for $\mathcal{F}_{K_3}(S_b,n)$ is $b=\Theta(n)$;
\item The threshold bias for $\mathcal{C}(S_b,n)$ is $b=\Theta(n)$;
\item The threshold bias for $\mathcal{H}(S_b,n)$ is $b=\Theta(n)$;
\end{enumerate}
\end{theorem}

Comparing these results to threshold biases of classical Maker-Breaker games, we observe that the structure restriction tends to be a major obstruction to Breaker's original winning strategy, even though matchings and stars are generally viewed as inherently distinct structures. These results illustrate that the structure-biased Maker-Breaker games deviate from their classical counterparts. Moreover, they shed light on approaching problems in classical Maker-Breaker games, in the sense that they reflect what the optimal Breaker should and should not look like. For example, item 1's from Theorem~\ref{big1},~\ref{big2},~\ref{big3} show that the double star structure in Breaker's known optimal strategy for the classical $\mathcal{F}_{K_3}(b,n)$ is crucial.

\section{Notation and Convention}
\begin{description}
\item$M_i$: Maker's graph after her $i$th move
\item$B_i$: Breaker's graph after his $i$th move
\item$d_M(v)$: the degree of $v$ in Maker's graph
\item$d_B(v)$: the degree of $v$ in Breaker's graph
\item$N_G(U)$: the set of neighborhood of $U$ in graph $G$ 
\end{description}
Throughout the paper, unless explicitly stated, we assume Breaker to take the first move in considering Maker strategies, and vice versa.

\section{$(1:K_m)$-Biased Maker-Breaker Games on $K_n$}
The first type of game we would like to consider is the structure-biased triangle game $\mathcal{F}_{K_3}(K_m,n)$, which is defined to have all the edges of $K_n$ as its board and all copies of $K_3$ to be its winning set, with bias for the game for being $1:K_m$, meaning Maker can take one edge per move while Breaker can take a copy of $K_m$ or a subgraph of $K_m$ per move. In Theorem~\ref{cliq-triangle}, we determine the threshold bias for $\mathcal{F}_{K_3}(K_m,n)$. First we recall a classical result:

\begin{theorem}
~~\cite{hefetz2014} 
\label{class-triangle}
Let $\mathcal{F}_{K_3}(b,n)$ denote the general biased Maker-Breaker triangle game with $(1:b)$ bias. Then Breaker can win with $b\geq 2\sqrt{n}$. 

 \end{theorem}

\begin{theorem}
\label{cliq-triangle}
Maker has a winning strategy for the game $\mathcal{F}_{K_3}(K_m,n)$ if $m\leq \sqrt{2n-4}$, while Breaker has a winning strategy if $m\geq 4\sqrt n$.
\end{theorem}
Note that the order of magnitude for the threshold bias of $\mathcal{F}_{K_3}(K_m,n)$ is different from the one of the classical triangle game $\mathcal{F}_{K_3}(b,n)$: In terms of number of edges, the former is $\Theta(n)$ while the latter is $\Theta(\sqrt{n})$.

\begin{proof}
We first consider the Maker strategy. We may assume Breaker makes the first move, claiming $K_m$ on $v_1,v_2,\dotsb,v_m$. At Maker's first turn, Maker chooses a vertex $v\notin \{v_i\}_{i=1}^m$ and arbitrarily claims an edge $vu_1$ incident to $v$. For every positive integer $i\geq 2$, if there exists an edge that, upon choosing, Maker would form a triangle in her graph, then Maker chooses this edge; otherwise Maker chooses a vertex $u_i$ not in Breaker's graph at the current position and claims the edge $vu_i$. We claim that Maker can win within $m+1$ turns: if Maker encloses a triangle within the first $m$ turns, then she wins the game; otherwise, at stage $i$ for every $i\geq 3$, based on Maker's strategy, Breaker has to include the vertices $u_1,u_2,\dotsb, u_{i-1}$ in his graph, because otherwise, if $u_j,j\in[i-1]$ is not claimed by Breaker in stage $i$, Maker can win by claiming $u_1u_{i-1}$ if $j=i-1$ and by claiming $u_ju_{i-1}$ if $j\neq i-1$. Now because $m\leq \sqrt{2n-4}$, we have that
\begin{equation}
2m+\sum_{i=3}^{m}(m-i+1)< n-1
\end{equation}

In the first two turns, Breaker can include at most $2m$ vertices in his graph; at the $i$th turn for $3\leq i\leq m$, by our analysis above, he can include at most $m-i+1$ vertices that are not in Maker's graph at that point of the game. Then $(1)$ implies that, at every turn there always exists a vertex $u$ adjacent to $v$ such that $d_B(u)=0$, and Maker can then choose this $u$ to be the $u_j$ and claim $vu_j$ at turn $m+1$ Maker claims the edge $vu_{m+1}$ where $u_{m+1}$ is not in Breaker's graph, and at Breaker's $(m+2)$th turn he can choose at most $m$ vertices among $u_1,u_2,\dotsb,u_{m+1}$ to form his graph. Let $u_j,j\in[m+1]$ be the vertex not claimed by Breaker at stage $m+2$. If $j=m+1$, then Maker wins by claiming $u_1u_{m+1}$; if $j\neq m+1$, then Maker wins by claiming $u_ju_{m+1}$.

Now we consider Breaker strategy. If $m\geq 4\sqrt{n}$, then Breaker has a winning strategy for the classical game $\mathcal{F}_{K_3}(\frac{m}{2},n)$ according to Theorem~~\ref{class-triangle}, which is possible since $\frac{m}{2}\geq 2\sqrt{n}$.  Then at each turn, Breaker can occupy a $K_m$ that contains all edges in this winning strategy for the classical game, and thus win $\mathcal{F}_{K_3}(K_m,n)$.
\end{proof}

The second game we would like to discuss is the connectivity game $\mathcal{C}(K_m,n)$, which is defined to have all the edges of $K_n$ as its board and all copies of spanning trees as its winning set, with the bias being $1:K_m$.
\begin{theorem}
    \label{clique-clique}
    The threshold bias for the game $\mathcal{C}(K_m,n)$ is $m=\Theta( \sqrt{\frac{n}{\ln n}}).$
\end{theorem}
We first prove a modified version of the winning condition for BoxMaker in the Box game:
\begin{lemma} Consider the game whose winning set are $q$ disjoint sets of same size $p$: $A_1,A_2,\dotsb,A_q$ with $|A_1|=|A_2|=\dotsb=|A_q|=p$. Within every move, BoxMaker can claim a rectangle of perimeter $m$, i.e. an $a\times b$ rectangle with $a+b\leq m$, while BoxBreaker can choose to delete any winning set. Then BoxMaker has a winning strategy if $(1-\frac{2}{m})^{\frac{2p}{m}}>\frac{2p}{mq}+\frac{1}{q}$.
\end{lemma}
\begin{proof}
    BoxMaker's strategy is to claim at every turn a $\frac{m}{2}\times \frac{m}{2}$ rectangle and balance the size of unclaimed elements in all remaining winning sets (i.e. minimizing deviation). Letting $a_i=(1-\frac{2}{m})^{i-1}\frac{2q}{m}$, $b_i=\ceil{a_i}$, then $b_i$ denotes the number of turns BoxMaker dwell on the $i$th column with width $\frac{m}{2}$. BoxMaker wins under this strategy if $\Sigma^{\ceil{\frac{2p}{m}}}_{i=1}b_i<q$, i.e. if she can claim a full row under $\ceil{\frac{2p}{m}}$ steps without BoxBreaker deleting all boxes. The condition in our assumption is sufficient for this to hold.
\end{proof}
\begin{proof}[Proof of Theorem~\ref{clique-clique}:]
We first consider the Breaker strategy. Breaker proceeds in two stages:
\begin{enumerate}
\item  Breaker creates a clique $G$ on vertex set $V$ with $|V|=\frac{m^2}{16}$ such that $G$ is isolated from Maker's graph in the following procedure:
Fix $\{V_i\}_{1\leq i\leq \frac{m+2}{4}}$ pairwise disjoint families of vertices of the board $K_n$ with each $|V_i|=\frac{m}{2}$. Breaker aims for a complete graph $H$ whose vertex set is the $\{V_i\}_{1\leq i\leq\frac{m+2}{4}}$: Within every move, Breaker claims an edge $e=\{V_i,V_j\}$ from $H$ by claiming the $m$ vertices in $V_i\cup V_j$. This process ends in $\frac{1}{2}\frac{m+2}{4}\frac{m-2}{4}=\frac{m^2-4}{32}$ steps, implying Maker can claim at most $2\frac{m^2-4}{32}=\frac{m^2-4}{16}$ vertices from Breaker's graph, and so Breaker has successfully claimed a subgraph of $G$, and thus of the board $K_n$ with at least $\frac{m}{2}\frac{m+2}{4}-\frac{m^2-4}{16}=\frac{(m+2)^2}{16}$ vertices.
\item Breaker aims to isolate a vertex $v$ in $G$ by claiming all edges adjacent to $v$. This is equivalent to assuming the role of BoxMaker in the modified box game of perimeter $m$, where $p=n-|V|=n-\frac{m^2}{16}$ and $q=|V|=\frac{m^2}{16}$. Applying the previous lemma renders the desired result on the upper bound.
\end{enumerate}
The lower bound is given by the known result on the classical connectivity game $\mathcal{C}(n)$ (Theorem~\ref{classic-con}) with bias $1:{m\choose 2}$.
\end{proof}
Note that the above order of magnitude for the threshold bias applies also for $(1:K_m)$-biased Hamiltonicity game $\mathcal{H}(K_m,n)$, as the strategy for Breaker of isolating a vertex applies to Hamiltonicity game, and the lower bound has been established to be the same order as in connectivity game (Theorem~\ref{classic-ham}).

\section{$1$-to-$b$-Matching-biased Maker Breaker Games on $K_n$}
The second bias structure we would like to consider is when Breaker is restricted to taking $m$ disjoint edges per move. Again we consider the game $\mathcal{F}_{K_3}(E_b,n)$ with $K_n$ as its board and all copies of $K_3$ being its winning set, with bias for the game being $1:E_b$ where $E_b$ denotes $b$ disjoint edges. It turns out that Maker can easily win the game within $4$ moves even when Breaker is allowed to take a maximal matching per move:
\begin{theorem}
    Breaker cannot win $\mathcal{F}_{K_3}(E_b,n)$ even with a maximal matching bias (i.e. $b=\lfloor\frac{n}{2}\rfloor$) for all $n\geq 10$. 
\end{theorem}
\begin{proof}
    We consider the Maker strategy under a maximal matching bias for Breaker: In her first $2$ moves, Maker claims $2$ edges sharing an end vertex $v$: $u_1v,u_2v$, which she can always do as long as $n\geq 4$. In her third move, Maker claims an edge $uv$ where neither $u_1u,u_2u$ is in Breaker's graph, which she can do for all $n\geq 10$. In the following Breaker's turn, Breaker could only claim at most one of $u_1u$ and $u_2u$, and Maker claiming the other would render her winning.
\end{proof}
We now look at the connectivity game. Here what we know depends on which player makes the first move.
\begin{proposition}
Breaker cannot win the connectivity game $\mathcal{C}(E_b,n)$ even for a maximal matching bias (i.e. $b=\lfloor\frac{n}{2}\rfloor$) if Maker makes the first move.
\end{proposition}
\begin{proof}
Maker's strategy is to maintain a connected graph and never take a cycle. Let $\Delta_B(i)$ denote the maximum degree of Breaker's graph, and $c_M(i)$  the size of Maker's graph by the end of the $i$th turn for $1\leq i\leq n-1$. Then $\Delta_B(i)\leq i$, and $c_M(i)=i+1$, and so we have $\Delta_B(i)<c_M(i)$ for each $i$. If Breaker wins at the $i$th turn for some $1\leq i\leq n-1$, then $\Delta_B(i)\geq c_M(i)$ which is a contradiction.
\end{proof}
The situation where Breaker goes first turns out to be more delicate. For even $n$, if Breaker goes first, then he can win when $b=\frac{n}{2}$ by decomposing $K_n$ into $n$ pairwise disjoint perfect matchings and take available edges of the matchings in any order in his $n$ turns, regardless of Maker's choices. If the game proceeds to the $n$th turn, Breaker will have taken all edges in the board not in $M_{n-1}$ on his $n$th move, and thus wins.

\begin{theorem}
\label{con-match}
The threshold bias $b$ for the connectivity game $\mathcal{C}(E_b,n)$ satisfy $$b>0.1 n$$
\end{theorem}
\begin{proof}
For Maker's strategy, we extend the method introduced in~\cite{gebauer2009}:

In the beginning of the game, let all vertices in the board be active. During the course of the game, a vertex is active if it satisfies both of the following:
\begin{enumerate}
    \item It is in a connected component of size less than $n-2b$ in Maker's graph;
    \item It has not been deactivated by Maker.
\end{enumerate} 

For every vertex $v$, define the danger function \[
dan(v) =
\begin{cases} 
d_B(v)&\text{if }v \text{ is active}; \\
0, & \text{otherwise}
\end{cases}
\]

At the $i$th turn, Maker's strategy is to choose an arbitrary vertex $v$ of the highest $dan(v)$, pick an arbitrary edge between the component $v$ sits in and another connected component, and deactivate $v$.

Following this strategy, Maker never forms a cycle in her graph, so assuming Breaker achieves a cut between vertex set $A$ and $B$ with $|A|\leq|B|$ in $g$ steps, we first have $g\leq n$. Also, Breaker could have taken at most $gb\leq nb$ edges, and so $|A|\leq 2b$, because otherwise $|A||B|> 2b(n-2b)>nb$. Since there is no Maker edge between $A$ and $B$, there exists some vertex $v_g\in A$ that is active, and its danger is at least $n-2b$. Let $C=\{v_1,v_2,\dotsb,v_{g-1},v_g\}$ where $v_1,v_2,\dotsb,v_{g-1}$ is the sequence of vertices deactivated by Maker. Let $C_i=\{v_i,v_{i+1},\dotsb,v_g\}$, and let $a(i)$ be the average danger of $C_i$ after the $i$th turn for $0\leq i\leq g$, where $i=0$ represents the beginning of the game. Then we have $a(0)=0$ since initially every vertex has danger $0$, while $a(g)\geq n-2b$ because $a(g)$ is equal to the danger of $v_g$ after the $g$th turn. Note that Maker's strategy guarantees that she never increases $a(i)$ for every $i$. We divide into two cases:

\emph{Case 1}: $g>b$

Because Breaker is only allowed to take disjoint edges, we have $a(i+1)\leq a(i)+\min\{\frac{2b}{i},1\}$. For $0\leq i\leq g-b$, we bound by $a(i+1)\leq a(i)+\frac{2b}{i}$; for $g-b<i<g$, we bound by $a(i+1)\leq a(i)+1$. Because $$a(0)= a(g)-\sum_{i=1}^g(a(i)-a(i-1)),$$by the analysis above, we have
$$0\geq n-2b-\sum_{i=b}^{g} \frac{2b}{i}-b\geq n-3b-2b \ln{n} +2b\ln{b}$$
which implies $b>0.1n$.

\emph{Case 2}: $g\leq b$
Then for $0\leq i<g$, Breaker can increase $a(i)$ by at most $1$, which implies $$0\geq n-2b-b$$
which implies $b\geq \frac{1}{3}n$.
\end{proof}


\begin{theorem}
\label{ham-match}
The threshold bias $m$ for the Hamiltonicity game $\mathcal{H}(E_b,n)$ is $\Theta(n)$.
\end{theorem}

We adopt a similar approach to that in~\cite{kriv2011}. A vertex $v$ is said to be \emph{active} if and only if $d_M(v)<16$. For every vertex $v$, define the danger function
\begin{equation}
\label{eq:dan} 
dan(v) =
\begin{cases} 
d_B(v)-2bd_M(v)&\text{if }v \text{ is active}, \\
0, & \text{otherwise}
\end{cases}
\end{equation}
Recall that the \textbf{Minimum Degree game} $\mathcal{D}_d(b,n)$ has $K_n$ as its board and all subgraphs of $K_n$ on all $n$ vertices of $K_n$ such that every has degree at least $d$.
\begin{lemma}
\label{md-match}
For every $\delta>0$, there exists $c>0$ such that whenever $b\leq cn$, by picking a vertex of largest danger and claiming a previously unclaimed edge adjacent to it, Maker can win the $(1:b)$ Minimum Degree Game $\mathcal{D}_{16}(E_b,n)$ in a way that during the play, whenever a vertex $v$ has Maker degree $d_M(v)<16$, then $d_B(v)<(1-\delta)n$.
\end{lemma}

\begin{proof}
We assume that after $g$ turns there exists a vertex $v_g$ such that $d_M(v_g)<16$ while $d_B(v_g)\geq (1-\delta)n$. Let $C=\{v_1,v_2,\dotsb,v_{g-1},v_g\}$ where $v_1,v_2,\dotsb,v_{g-1}$ is the sequence of vertices eased by Maker. Note that here $v_i$ could be the same vertex as $v_j$ even if $i\neq j$. Let $C_i=\{v_i,v_{i+1},\dotsb,v_g\}$ (as a set instead of a multiset), and let $a(i)$ be the average danger of $C_i$ after the $i$th turn for $0\leq i\leq g$, where $i=0$ represents the beginning of the game. Then $a(i+1)>a(i)$ can only occur if $|C_{i+1}|<|C_i|$ by the definition of danger \eqref{eq:dan}. Again, we divide into two cases.

\emph{Case 1}: $|C|>b$

If $i$ is such that $|C_{i+1}|<|C_i|$, then Maker does not increase $a(i)$ while Breaker increases $a(i)$ by at most $\min \{\frac{2b}{|C_i|},1\}$. For $|C_i|>b$, we bound by $a(i+1)\leq a(i)+\frac{2b}{|C_i|}$; for $|C_i|\leq b$, we bound by $a(i+1)\leq a(i)+1$. 
$$0\geq (1-\delta)n-30b-2b(\ln n-\ln b)-b$$
which implies $b\geq cn$ where $c$ is the unique solution in $(0,1)$ to $31c-2c\ln c=1-\delta$.

\emph{Case 2}: $|C|<b$

If $i$ is such that $|C_{i+1}|<|C_i|$, then Maker does not increase $a(i)$ while Breaker increases $a(i)$ by at most $1$, and so

$$0\geq (1-\delta)n-30b-b$$
which implies $b\geq\frac{1-\delta}{31}n$.

\end{proof}

Before we proceed to the proof of Theorem~\ref{ham-match}, we recall some definitions and facts:

\begin{definition}A graph $G=(V,E)$ is a \textbf{$k$-expander} if for every $U\subseteq V$ with $|U|\leq k$ we have that $|N_G(U)|\geq 2|U|$.
\end{definition}

\begin{definition}
Let $G$ be a graph. A non-edge $e$ is a \textbf{booster} if $G\cup e$ is Hamiltonian or if the longest path in $G\cup e$ is longer than the one in $G$.
\end{definition}

\begin{proposition}
\label{exp-almost-con}
(Folklore; see~\cite{hefetz2014}) Every connected component in a $k$-expander has size at least $3k$.
\end{proposition}

\begin{theorem}
\label{con-exp}
(Folklore; see~\cite{hefetz2014}) A connected $k$-expander has at least $\frac{{(k+1)}^2}{2}$ boosters.
\end{theorem}

\begin{proof}[Proof of Theorem~\ref{ham-match}:]
\leavevmode\\

\emph{Stage 1}:

Maker's strategy is to pick a vertex of largest danger and claim a random unclaimed edge adjacent to it; we call this strategy \emph{strategy $S$}.  In~\cite{kriv2011} it was proved that

\begin{proposition}\label{expander}
~\cite{kriv2011}
For sufficiently small $\delta,$ if it is true that Maker, by always claiming an arbitrary edge incident with a vertex of maximum danger, will win $\mathcal{D}_{16}(E_b,n)$ and also ensure that anytime during the play, $d_B(v)<(1-\delta)n$ whenever $d_M(v)<16$, then Maker will create a $k$-expander for $k=\delta^5n$ within $16n$ moves by making these arbitrary choices uniformly at random, i.e. by following strategy $S$.

\end{proposition}
Fix a small $\delta$ as in Proposition~\ref{expander}, and let $c$ be the corresponding constant from Lemma~\ref{md-match}. Let $c_0$ be such that $c_0<\min\{\frac{\delta^{10}}{36},c\}$ and $b=c_0n$. Then by the end of this stage, Lemma~\ref{md-match} and Proposition~\ref{expander} guarantee that Maker has achieved a $k$-expander.

\emph{Stage 2}:

Now at this point each of the connected component in Maker's graph has size at least $3k$ by Proposition~\ref{exp-almost-con}, and Maker can turn it into a connected graph within $\lfloor\frac{n}{3k}\rfloor$ moves; if she does so, Stage 2 ends: During this stage, Breaker's graph has size at most $17n\cdot c_0n$, while Breaker needs at least $(3k)^2$ many edges to achieve a cut between two of the connected components in Maker's graph. But $17n\cdot c_0n<(3k)^2$, thus Breaker cannot achieve such a cut and Stage 2 ends successfully for Maker.

\emph{Stage 3}:

By Theorem~\ref{con-exp}, throughout this stage there exist $\frac{k^2}{2}$ available boosters for Maker, and if Maker keeps picking boosters, she wins the game within $n$ moves. Because $18n\cdot c_0n<\frac{k^2}{2}$, Breaker cannot stop Maker from carrying out this process, and Maker wins the game.

\end{proof}

\section {$1$-to-$b$-Star-biased Maker Breaker Games on $K_n$}
The third bias structure we study is star $S_b$ defined as $b$ edges with a shared vertex, call it the center. The game $\mathcal{F}_{K_3}(S_b,n)$ has $K_n$ as its board and all copies of $K_3$ its winning set, with bias for the game being $1:S_b$.
\begin{theorem}
    Maker has a winning strategy for the game $\mathcal{F}_{K_3}(S_b,n)$ when $b< \frac{{{n-1}\choose 2}}{n}$.
\end{theorem}
\begin{proof}
 Assume Breaker takes $\{uv_1,uv_2,\dotsb,uv_b\}$ in his first move. Maker in her first move claims an edge $w_1w_2$ where $w_1,w_2\notin\{u,v_1,\dotsb,v_b\}$. In Breaker's second move, he can choose at most one of $w_1,w_2$ as the center, and we denote the other one as $w$. Maker's strategy starting from her second move is as follows: at the beginning of Maker's $i$th turn in the game, if Maker could achieve a triangle by claiming an unoccupied edge, Maker chooses this edge and wins; otherwise assume Maker's current graph consists of $\{wx_1,wx_2,\dotsb,wx_{i-1}\}$, and Maker chooses a vertex $x_i$ such that $|\{x_ix_j\notin E(B_i):j\in\{1,2,\dotsb,i-1\}\}|\geq 1$, and claims the edge $wx_i$. If no such $x_i$ exists, Maker forfeits the game. Under this strategy, Breaker is forced to put the center of his next move at $x_i$ or $x_j$ for some $x_ix_j\notin E(B_i)$, and so fails to claim any free edge adjacent to $w$. Now we claim that Maker, if she hasn't won by then, would be able to claim all edges incident to $w$ without forfeiting the game: For every $i\leq n-1$, right before Maker's $i$th move, $|E(B_i)|\leq ib$, and since Maker hasn't won by then, $\{x_jx_k:j,k\leq i-1\}\subseteq E(B_i)$, and thus at most $ib- {i-1\choose 2}$ edges in $B_i$ are between the sets $\{x_j:j\leq i-1\}$ and $V(K_n)\setminus \big(\{w\}\cup\{x_j:j\leq i-1\}\big)$. But the total number of edges between these two vertex sets is $i(n-i)$. Now $$ib-{i-1\choose 2}-i(n-i)=i(b+\frac{i+3}{2}-n-\frac{1}{i})$$
and $$b+\frac{i+3}{2}-n-\frac{1}{i}< b+\frac{n+2}{2}-n<\frac{n-2}{2}+\frac{n+2}{2}-n=0$$
 where the second inequality follows from $b<\frac{{{n-1}\choose 2}}{n}$. Then we have that $$ib-{i-1\choose 2}-i(n-i)<0$$which proves our claim. Now right before Maker's $n$th move, because $|E(B_n)|\leq nb<{{n-1}\choose2}=|\{x_ix_j:i,j\in\{1,2,\dotsb,n-1\}\}|$, there exists $x_kx_l,k,l\in\{1,2,\dotsb,n-1\}$ such that $x_kx_l\notin E(B_n)$. Maker then wins the game by claiming $x_kx_l$.
    
\end{proof}
Before proceeding to the connectivity game, we establish a lemma which is interesting on its own:

\begin{lemma}
\label{happylem}
Let $k,c\geq 2$ be positive integers, and $a_1,a_2,\dotsb,a_k$ be an arbitrary sequence of real numbers with average $a$. On the $i$th turn for $1\leq i\leq k-1$, one adds $c$ to an arbitrary surviving member of the $a$'s, adds $1$ to each of the rest of the surviving members, and deletes the largest among them. Let $a'$ denote the value of the remaining number after the $(k-1)$th turn. Then $$a'-a\leq c+k-\frac{c-1}{k}-2$$Moreover, this bound is tight.
\end{lemma}

\begin{proof}
After suitable translation, we may assume $a=0$. Let $m(0)=\max\{a_1,a_2,\dotsb,a_k\}$, and for every $1\leq i\leq k-1$ let $m(i)$ be the maximum among the surviving members after the $i$th turn, and let $d(i)$ be the deleted number at the $i$th turn. Then we observe that
\begin{enumerate}
\item $m(k-1)=a'$;
\item $m(i-1)+1\leq d(i)\leq m(i-1)+c$;
\item $d(i)\geq m(i)$.
\end{enumerate}

Now we claim that, for every $1\leq i\leq k-1$, $m(i)\leq m(i-1)+1$: This is because after the adding operation in the $i$th turn, at most one member is greater than or equal to $m(i-1)+2$, which would then be deleted.

Thus, by observation 1 and the claim,
\begin{equation}
m(i)\geq a'-k+i+1\label{meow1}
\end{equation} for every $0\leq i\leq k-1$. Since $a=0$, we have by double counting that \begin{equation}
    \big(\sum_{i=1}^{k-1}d(i)\big)+a'=\sum_{i=1}^{k-1}(c+k-i)\label{meow2}
\end{equation}
The right hand side simplifies to $c(k-1)+\frac{1}{2}k(k-1)$, whereas \begin{equation}
\Big(\sum_{i=1}^{k-1}d(i)\Big)+a'\geq \Big(\sum_{i=1}^{k-1}m(i-1)+1\Big)+a'\geq \Big(\sum_{i=1}^{k-1}a'-k+i+1\Big)+a'=a'k-\frac{1}{2}(k-1)(k-2)\label{meow3}
\end{equation}where the first inequality follows from observation 2, and the second inequality follows from ~\eqref{meow1}. Combining ~\eqref{meow2} and ~\eqref{meow3} we get
\begin{equation}
a'k\leq c(k-1)+(k-1)^2\notag
\end{equation}
which simplifies to 
\begin{equation}
a'\leq c+k-2-\frac{c-1}{k}\notag
\end{equation}

To show that the bound is tight, consider the sequence $a_1=a_2=\dotsb=a_{k-1}=0, a_k=c-1$, then $a=\frac{c-1}{k}$. For every $1\leq i\leq k-1$, add $c$ to the element $a_i$ and $1$ to each of the other surviving elements. Then $a'=(c-1)+(k-1)=c+k-2$, and we have
\begin{equation}
a'-a=c+k-2-\frac{c-1}{k}\notag
\end{equation}
\end{proof}

We will use the previous lemma to establish the lower bound in the next theorem:

\begin{theorem}
    Let $m$ denote the threshold bias for $\mathcal{C}(S_b,n)$. Then $$0.1n\leq b\leq \frac{3n}{7}$$
\end{theorem}
\begin{proof}
We first consider Breaker's strategy. Let $b=\frac{3n}{7}$. Breaker moves in two stages:

Stage 1: Breaker aims to create a set $T$ of size $t=8b-3n=\frac{3n}{7}=b$ vertices that are pairwise non-adjacent in Breaker's graph such that for every $v\in T$
\begin{enumerate}
\item the maker degree $d_M(v)=0$;
\item the breaker degree $d_B(v)\geq n-2b$.
\end{enumerate}
To achieve this goal, Breaker fixes two subsets $A,B$ of the board $K_n=(V,E)$ such that:

\begin{enumerate}
\item $|A|=|B|=b$;
\item $A\cap B=\emptyset$.
\end{enumerate}

Then $|V\setminus (A\cup B)|=n-2b$. Enumerate $V\setminus (A\cup B)$ as $v_1,v_2,\dotsb,v_{n-2b}$. In the first turn, Breaker takes all available edges of the star centered at $v_1$ with $A$ being its leaves. In the second turn, Breaker takes the star centered at $v_1$ with $B$ being its leaves. Now if before his third turn, Maker hasn't claimed an edge adjacent to $v_1$, Breaker takes the star centered $v_1$ with the rest of its neighbors as the leaves, and Breaker wins by isolating $v_1$; otherwise Breaker takes all available edges of the star centered at $v_2$ with $A$ being its leaves, and then Breaker continues this process. The process can continue to at most $2(n-2b)$ turns (if Breaker hasn't won by then), then among $A\cup B$ there are at least $2b-2\cdot 2(n-2b)+{n-2b}=8b-3n$ vertices that have degree $0$ in Maker's graph: $|A\cup B|=2b$, and among the $2(n-2b)$ turns, Maker can include at most $2\cdot 2(n-2b)$ vertices in her graph, which should contain $V\setminus(A\cup B)$ to make sure Breaker doesn't isolate a vertex among $V\setminus(A\cup B)$. Then this set of $8b-3n$ vertices satisfies Breaker's goal for Stage 1.

Stage 2: Enumerate $T$ as $u_1,u_2,\dotsb,u_{t}$ where $t=\frac{3n}{7}$. On the $i$th turn of Stage 2 for each $i<t$, Breaker takes the star centered at $u_i$ with leaves including $u_{i+1},u_{i+2},\dotsb,u_t$. Then Maker has to respond by occupying an edge adjacent to $u_i$ because otherwise Breaker wins by isolating $u_i$ on his next move; this edge can only be between $u_i$ and $V\setminus T$ because at this point $u_iu_j$ has been taken by Breaker for all $j\neq i$. This also implies that, before Breaker's $t$th turn, $d_M(u_t)=0$. On the $t$th turn, Breaker takes the star centered at $u_{t}$ with the leaves being the rest of its unoccupied neighbors. Then since $$t-1+(n-2b)+b=8b-3n-1+n-2b+b\geq n-1$$Breaker wins by isolating $u_{t}$: Right before the $t$th turn, $u_t$ is adjacent at least $n-2b$ vertices other than the $u_i$'s by how Breaker claimed $T$, and $u_t$ is also adjacent to $u_i$ for $1\leq i\leq t-1$, and so there are only $b$ neighbors of $u_t$ left to claim, which can be occupied by Breaker in his $t$th turn.

We next establish the lower bound. For this we use the same Maker strategy as in Theorem~~\ref{con-match} and keep the same definitions and setups.

Again, we note that Maker does not increase $a(i)$ for every $i$. We divide into two cases as before.

\emph{Case 1}: $g>b$

For $0\leq i\leq g-b$, Breaker can increase $a(i)$ by at most $\frac{2b}{i}$ and so we have $a(i+1)\leq a(i)+\frac{2b}{i}$. Now for the last $b$ rounds, we apply Lemma~\ref{happylem} with $k=b$ and $c=b$ to conclude that $a(g)\leq a(g-b)+2b+b$. Then we have

$$0\geq n-2b-\sum_{i=b}^{g} \frac{2b}{i}-3b\geq n-5b-2b \ln{n} +2b\ln{b}$$
which implies $b>0.1 n$.

\emph{Case 2}: $g\leq b$

Apply Lemma 5.2 with $k=g$ and $c=b$, we get that $a(g)\leq a(0)+b+g$. Thus,

$$0\geq n-2b-2b$$
which implies $b\geq \frac{1}{4}n$.

\end{proof}

\begin{theorem}
\label{ham-star}
The threshold bias for Hamiltonicity game $\mathcal{H}(S_b,n)$ is $\Theta(n)$.
\end{theorem}
The proof is exactly the same as in Theorem~\ref{ham-match} except we need to establish an equivalent of Lemma~\ref{happylem}. To achieve this, we show

\begin{lemma}
\label{happybutannoying}
Let $c\geq k\geq 2$ be positive integers, and $(a_1,t_1),(a_2,t_2),\dotsb,(a_k,t_k)$ is an arbitrary sequence of pairs where $(a_i)_{i=1}^k$ is a sequence of real numbers with average $\bar{a}$, and $(t_i)_{i=1}^k$ is a sequence of natural numbers such that $t_i<16$ for each $i$. For every turn, one adds $c$ to an arbitrary surviving member of the $a$'s, adds $1$ to each of the rest of the surviving $a$'s. One then picks the pair $(a,t)$ with the largest $a$ and increase $t$ by $1$. If $t=16$, then one deletes $(a,t)$ from the sequence; otherwise one decreases $a$ by $2c$. Let $(a',t')$ denote the last remaining tuple when there is only one tuple left. Then $$a'-\bar a\leq c+17k$$
\end{lemma}
\begin{proof}
Again we may assume $\bar a=0$. For $0\leq i\leq k$ we let $f(i)$ be the maximum among the $a$'s right after $i$ tuples have been deleted; for $1\leq i\leq k$ we let $d(i)$ be the value of the first coordinate of the $i$th deleted tuple; for $0\leq i\leq k-1$ we let $h(i)$ be the number of turns after $i$ tuples have been deleted but before $i+1$ tuples have been deleted, $g(i)$ be the maximum among the $a$'s right before $i+1$ tuples have been deleted. Note that $h(i)$ could be $0$, in which case we have $f(i)=g(i)$. The relationship and order are illustrated in the following:
\begin{center}
\begin{tikzcd}[column sep=large]
f(0) \arrow[r, "{\scriptstyle h(0)\text{ turns}}"] & 
g(0) \arrow[r, "{\scriptstyle \text{delete }d(1)}"] & 
f(1) \arrow[r, "{\scriptstyle h(1)\text{ turns}}"] & 
g(1) \arrow[r, "{\scriptstyle \text{delete } d(2)}"] & 
\cdots \arrow[r, "{\scriptstyle h(k{-}1)\text{ turns}}"] & 
g(k-1) \arrow[r, "{\scriptstyle \text{delete }d(k)}"] & 
f(k)=a'
\end{tikzcd}
\end{center}

Arguing as in Lemma~\ref{happylem}, we observe that:
\begin{enumerate}
\item $g(i)\leq f(i)+h(i)$ for $0\leq i\leq k-1$;
\item $d(i)\geq g(i-1)+1$ for $1\leq i\leq k$;
\item $f(i)\leq g(i-1)+1$ for $1\leq i\leq k$
\end{enumerate}

In addition, we have that $\sum_{i=0}^{k-1}h(i)\leq 16k$. Using this and Observation 1 and 3 we deduce that $g(i)\geq a'-16k-(k-1)+i=a'-17k+i+1$, and so we have $d(i)\geq a'-17k+i+1$ for $1\leq i\leq k$ by Observation 2. Then since $$\sum_{i=1}^{k-1}(c+i)-\sum_{i=1}^k d(i)-a'\geq 0$$we deduce that $$\sum_{i=1}^{k-1}(c+i)-\sum_{i=1}^k (a'-17k+i+1)-a'\geq 0$$which implies $a'\leq c+17k-19+\frac{19-2c}{k+1}$.
\end{proof}

Now we establish the analogue of Lemma~\ref{md-match}.

\begin{lemma}
\label{md-star}
For every $\delta>0$, there exists $c>0$ such that whenever $b\leq cn$, by picking a vertex of largest danger and claiming a previously unclaimed edge adjacent to it, Maker can win the $(1:b)$ Minimum Degree Game $\mathcal{D}_{16}(S_b,n)$ in a way that during the play, whenever a vertex $v$ has Maker degree $d_M(v)<16$, then $d_B(v)<(1-\delta)n$. 
\end{lemma}
\begin{proof}
We maintain the same definitions as in Lemma~\ref{md-match}.

\emph{Case 1}: $|C|>b$

For $|C_i|>b$, if $i$ is such that $|C_{i+1}|<|C_i|$, then Maker does not increase $a(i)$ while Breaker increases $a(i)$ by at most $\frac{2b}{|C_i|}$, and so we may bound by $a(i+1)\leq a(i)+\frac{2b}{|C_i|}$. Let $i$ be the smallest positive integer such that $|C_i|\leq b$. We apply the previous lemma to get that $a(g)-a(i)\leq 18b+16$ and so we have $$
0\geq (1-\delta)n-30b-2b(\ln n-\ln b)-18b-16$$
which implies $b\geq cn$ where $c$ is the unique positive solution in range $(0,1)$ to $49c-2c\ln c=1-\delta$.

\emph{Case 2}: $|C|<b$

We may directly apply the lemma to get $a(g)-a(0)\leq b+17g+16\leq 18b+16$ and so

$$0\geq (1-\delta)n-30b-18b-16$$
which implies $b\geq\frac{1-\delta}{49}n$.

\end{proof}
We may then finish the proof of Theorem~\ref{ham-star} exactly as in Theorem~\ref{ham-match}.

\section{Open Problems}
Other than closing the gaps in our theorems, there remain several interesting problems and directions that need future work:
\begin{enumerate}
\item There has been extensive work on establishing threshold biases for classical Maker-Breaker games on boards other than $K_n$ (See~\cite{hefetz2014}). It is thus interesting to ask about threshold biases for structure-biased Maker-Breaker games on different boards;
\item Different structures could lead to distinct threshold biases, which in turn require inherently different Breaker strategies. Thus it is interesting to study the threshold biases under other structures as well; in particular, our proofs for Theorem~\ref{con-match} and Theorem~\ref{ham-match} could be generalized to structures with constant degree, and so the definition for the structure biases can be relaxed to a wider family.

\end{enumerate}

\end{document}